\documentclass[11pt]{amsart}
\usepackage{amssymb,amsmath,amsthm}
\usepackage[hmargin=3cm,vmargin=3.5cm]{geometry}

\usepackage{graphicx}
\usepackage{amscd}
\usepackage[all, knot]{xy}
\usepackage{epsfig}
\usepackage{psfrag}
\usepackage{amssymb,amsmath,amsthm}
\usepackage{amsfonts}
\usepackage{amssymb}

\newtheorem{theorem}{Theorem}[section]
\newtheorem{lemma}[theorem]{Lemma}

\newtheorem{defn}[theorem]{Definition}
\newtheorem{conjecture}[theorem]{Conjecture}
\newtheorem{proposition}[theorem]{Proposition}
\newtheorem{prop}[theorem]{Proposition}
\newtheorem{cor}[theorem]{Corollary}

\newtheorem{example}[theorem]{Example}

\newcommand\A{{\mathcal A}}

\newcommand\Z{{\mathbb{Z}}}
\def\ra{\rightarrow}
\def\tor{\mathrm{tor}}
\def\mod{\mathrm{mod}}
\def\rank{\mathrm{rank}}
\def\free{\mathrm{free}}
\title{Torsion in Khovanov homology of semi-adequate links}

\author{Jozef H. Przytycki, Radmila Sazdanovi\'c}
\begin{document}
\maketitle

%\tableofcontents

\begin{quotation}
ABSTRACT. \baselineskip=10pt
The goal of this paper is to address A.~Shumakovitch's conjecture about the existence of $\Z_2$-torsion in
Khovanov link homology. We analyze torsion in Khovanov homology of semi-adequate links via chromatic cohomology
for graphs which provides a link between the link homology and well-developed theory of Hochschild homology.
In particular, we obtain explicit formulae for torsion and  prove that Khovanov homology of semi-adequate links contains $Z_2$-torsion
if the corresponding Tait-type graph has a cycle of length at least $3$. Computations show
that torsion of odd order exists but there is no general theory to support these observations.
We conjecture that the existence of torsion is related to the braid index.

\end{quotation}
\tableofcontents
%\newpage

%%%%%%%%%%%%%%%%%%%%
%%
%%   INTRODUCTION
%%
%%%%%%%%%%%%%%%%%%%%%
\section{Introduction}

In his visionary paper M.~Khovanov \cite{Kh-1} revolutionized the theory of quantum knot invariants by
categorifying the Jones polynomial of links. In 2003 A.~Shumakovitch conjectured
that any link which is not a connected or disjoint sum of Hopf links and trivial links
has torsion in Khovanov homology  \cite{Sh1, Sh2}.

In this paper we consider Khovanov homology of adequate and semi-adequate knots and links.
Adequacy is a natural generalization of the alternating property suitable for studying Khovanov homology.
Firstly, the outermost Khovanov homology group of  $+$-adequate links is equal to $\Z$, i.e.,
 $H_{n,*}(D)=H_{n,n+2|D_{s_+}|}(D)=\Z$, where $D$ is a
$+$-adequate diagram of a link $L$ with $n$ crossings.
Furthermore, M.~Asaeda and J.~Przytycki show that the next non-trivial homology group
$H_{n-2,n+2|D_{s_+}|-4}(D)$ has nontrivial
$Z_2$-torsion \cite{A-P} as long as the graph $G(D)=G_{s_+}(D)$ associated to the state $s_+$ is not
bipartite (see Section $2.1$ for definitions). In \cite{PPS} we explicitly compute $H_{n-2,n+2|D_{s_+}|-4}(D)$ of
$+$-adequate links showing, in particular, that for a non-split $+$-adequate diagram $D$
\begin{equation}
  \tor H_{n-2,n+2|D_{s_+}|-4}(D)=
\left\{
\begin{array}{ll}
    \Z_2, & \hbox{for $G(D)$ having an odd cycle;} \\
    0, & \hbox{{for $G(D)$ a bipartite graph.}}
\end{array}
\right.
\end{equation}

Torsion that lies in Khovanov homology one step deeper,
$H_{n-4,n+2|D_{s_+}|-8}(D)$, is analyzed in \cite{A-P}. In particular the authors show that for a strongly
$+$-adequate diagram $D$ with the graph $G(D)$ containing an even cycle
$H_{n-4,n+2|D_{s_+}|-8}(D)$ contains $\Z_2$-torsion. This statement implies Shumakovitch's
result that any alternating link which is not a connected or disjoint sum of trivial links and
Hopf links, has a nontrivial $Z_2$-torsion in its Khovanov homology.

In Section \ref{A2comultiplication} we compute the entire $H_{n-4,n+2|D_{s_+}|-8}(D)$ for many
classes of $+$-adequate diagrams, including
strongly $+$-adequate diagrams. In particular, we prove that for a $+$-adequate diagram $D$
\begin{equation}
  \tor H_{n-4,n+2|D_{s_+}|-8}(D)=
\left\{
\begin{array}{ll}
    \Z_2^{p_1(G'(D))-1}, & \hbox{for $G'(D)$ having an odd cycle;} \\
    \Z_2^{p_1(G'(D))}, & \hbox{{for $G'(D)$ a bipartite graph,}}
\end{array}
\right. \label{MainKh}
\end{equation}
 where $G(D)= G_{s_+}(D)$ is the graph associated to the Kauffman state $s_+$ and $G'(D)$
is a simple graph obtained from $G(D)$ by replacing multiple edges by singular edges (compare Section \ref{Background}).

In Section \ref{Background} we provide an overview of relations between plane graphs and link diagram, and the corresponding
polynomial invariants: the Kauffman bracket polynomial and the Kauffman bracket version of the Tutte polynomial.
Next we outline the theory of Khovanov homology, categorification of the Kauffman bracket polynomial, and a related
comultiplication free version of homology of graphs (derived by Helme-Guizon and Rong as a categorification of the
chromatic polynomial).

In Section \ref{DeeperA2} we prove Main Lemma \ref{Main Lemma} computing homology $H_{1,v-2}(G)$,  and derive
important corollaries.

In Section \ref{A2comultiplication} we modify translation of Khovanov homology to graph homology by allowing
one comultiplication. This allows us to compute torsion of $H_{n-4,n+2|D_{s_+}|-8}(D) $ for any
$+$-adequate diagram $D$.% (OR we change convention and work with $-$ adequate?)

In Section \ref{Braids} we give examples of adequate diagrams in a braid form starting from the $3$-braid
$\sigma_1^3\sigma_2^3 \sigma_1^2\sigma_2^2$ representing the knot $10_{152}$.

Finally in Section \ref{conj} we speculate about the existence of arbitrary torsion in Khovanov homology and
the relations to the braid index.
%we speculate about the shape of homology $H^{i,v-i}$ for
%a graph $G$ with $i \leq \ell(G).$

\section{Background}
\label{Background}

When developing our results on graph homology, we had in mind the application to Khovanov homology
of links. This is also the reason that we modify the comultiplication free version of Khovanov
homology of graphs introduced in \cite{HR} by allowing ``first" co-multiplication (Section \ref{A2comultiplication}). Thus we approximate
Khovanov homology one step further but still have homology of graphs independent on a surface embedding.

In this section we provide the background material: connection between graphs and links used in this paper. We also
recall relations between graph and link polynomials and between Khovanov homology and
its comultiplication free version for graphs.

\subsection{State graphs, state diagrams and the Kauffman bracket polynomial}\

Tait was the first to notice the relation between knots and
planar graphs. He colored the regions of the knot diagram alternately white
and black (following Listing) and constructed the graph by placing a vertex
inside each white region, and then connecting vertices by edges going through
the crossing points of the diagram
%(see Figure 1.4).  \cite{D-H}.

To generalize Tait construction and associate to any Kauffman state a graph we have to recall
some preliminary definitions.

\begin{defn}
A Kauffman state $s$ of $D$ is a function from the set of crossings of $D$ to
the set $\{+1,-1\}$. Diagrammatically, we assign to each crossing of $D$ a marker
according to the convention on Figure \ref{smooth}:
\begin{center}
  \begin{figure}[h]\begin{center}
    \scalebox{0.8}{\includegraphics{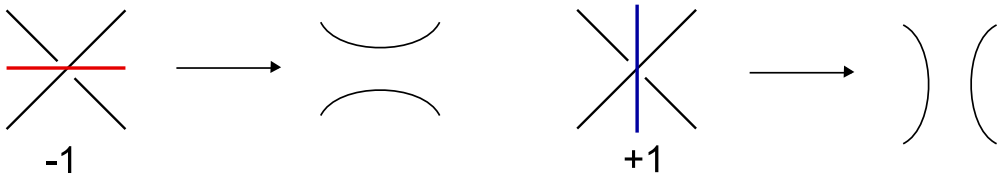}}
    \caption{Markers and associated smoothings.}\label{smooth}\end{center}
  \end{figure}
\end{center}

By $D_s$ we denote the system of circles in the diagram obtained by smoothing all
crossings of $D$ according to the markers of the state $s$, for example see Figure
\ref{WhG}b. Let $|D_s|$ denote the number of circles in state $D_s$.
\end{defn}

\begin{defn}\cite{PPS}\label{Definition PPS}\ \
Let $D$ be a diagram of a link and $s$ its Kauffman state. We form
a graph, $G_s(D)$, associated to $D$ and $s$ as follows.
 Vertices of $G_s(D)$ correspond to circles of $D_s$.
% and edges of $G_s(D)$ correspond to crossings of $D$.
Edges of $G_s(D)$ are in bijection with crossings of $D$ and an
edge connects given vertices if the corresponding crossing
connects circles of $D_s$ corresponding to the
vertices, see Figures \ref{WhG}, \ref{SmallAdequateNoTor}, \ref{SmallAdequateTor}, \ref{Link421}. As in the case of the Tait graph, $G_s(D)$
can be turned into a signed graph, with the sign of an edge $e(p)$  associated with the crossing $p \in D$  equal to the sign of the marker of the Kauffman state $s$ at that crossing $p$ (notice that we will not be working with
signed graphs in this paper).
\end{defn}

The Kauffman bracket polynomial of a diagram $D$,
  $<D>_{(\mu,A,B)}\in Z[\mu,A,B]$ is defined by:
  \begin{enumerate}
    \item $<U_n>= \mu^{n-1}$, where $U_n$ is the trivial diagram of $n$ components.
    \item $\langle D_{\parbox{0.5cm}{\psfig{figure=L+nmaly.eps}}}
\rangle =
A \langle D_{\parbox{0.5cm}{\psfig{figure=L0nmaly.eps}}}
\rangle + B \langle
D_{\parbox{0.5cm}{\psfig{figure=Linftynmaly.eps}}}
\rangle$
  \end{enumerate}

From this we obtain the state sum formula:
$$<D>_{(\mu,A,B)} = \sum_s A^{|s^{-1}(1)|}B^{|s^{-1}(-1)|}\mu^{|D_s|-1}.$$
In order to have invariance of the Kauffman bracket polynomial under regular isotopy
(i.e. Reidemeister moves $R_2$ and $R_3$)
we need $B=A^{-1}$ and $\mu=-A^2-A^{-2}$.

In this notation the Kauffman bracket polynomial of $D$ is given by
the state sum formula:
$<D> = \sum_s A^{\sigma(s)}(-A^2-A^{-2})^{|D_s|-1}$,
where $\sigma(s)= |s^{-1}(1)|- |s^{-1}(-1)| = \sum_p s(p)$ is the number
of positive markers minus the number of negative markers in the state $s$.\\

The unreduced Kauffman bracket $[D]$ is defined as $[D]= (-A^2-A^{-2})<D>$, thus:
\begin{equation*}
[D]= \sum_s A^{\sigma(s)}(-A^2-A^{-2})^{|D_s|}.
\end{equation*}

\begin{center}
  \begin{figure}[h]\begin{center}
    \scalebox{0.7}{\includegraphics{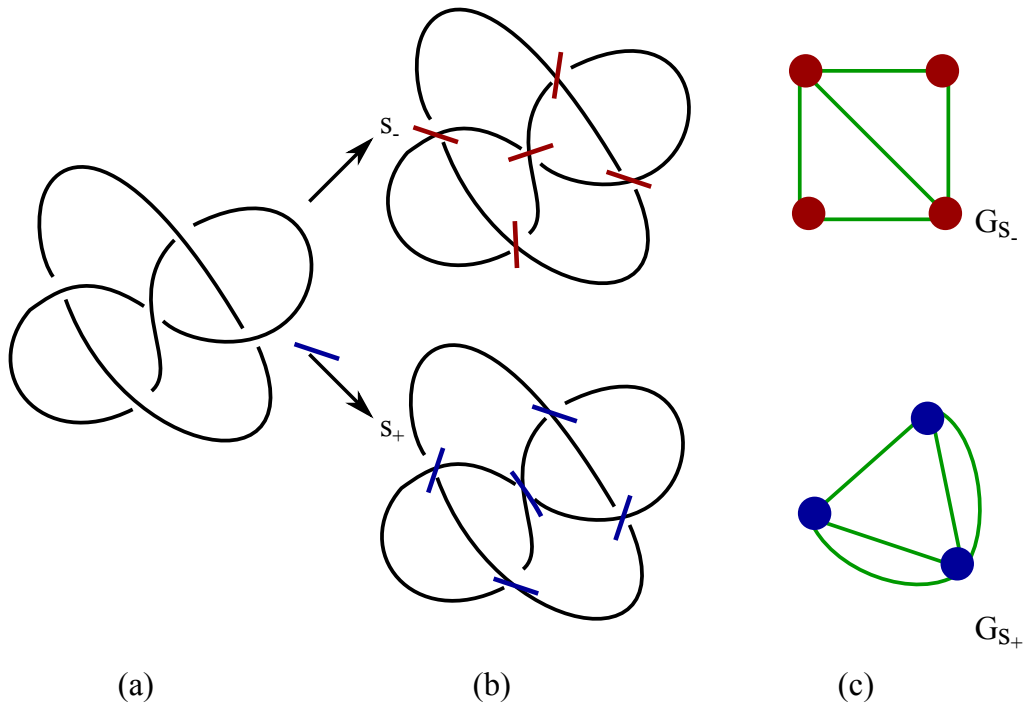}}
    \caption{ A minimal diagram of the Whitehead link (a), $D_{s_-}$ and $D_{s_+}$ in (b),
and the corresponding graphs $G_{s_-}$ and $G_{s_+}$ in (c).}
    \label{WhG}\end{center}
  \end{figure}
\end{center}

Before we move to the polynomial invariants of graphs we describe classes of knots and links we
will be analyzing in this paper, and classify corresponding graphs.

\begin{defn}\label{Definition PPS1}\ \
\begin{enumerate}
\item[(i)] In the language of graphs, the diagram $D$ is s-adequate
 if the graph $G_s(D)$ has no loops.
 \item[(ii)] The girth of a state $\ell(s)$ is the girth of the graph $G_s(D)$,  i.e. the length of the shortest cycle in $G_s(D)$ (in the
case $G$ is a forest we define $\ell(s)=\infty$). Thus $D$ is s-adequate if $\ell(s) > 1$. Similarly,
$D$ is strongly $s$-adequate if $\ell(s) > 2$.
\item[(iii)] The $s_+$ Kauffman state is a constant function sending all crossings to $+1$, and $s_-$ to $-1$.
We say that $D$ is $+$-adequate if it is $s_+$-adequate and that is $-$-adequate
if it is $s_-$-adequate\footnote{We follow notation from  \cite{A-P,H-P-R,PPS} in our notation.
In particular, if $D$ is an
alternating diagram then $G(D)$ is a signed Tait graph of $D$ with all negative edges.
However, we do not use signed graphs in this paper so our convention should not lead to confusion.
In this paper, generally $G(D)= G_{s_+}(D)$, and $+$-adequate diagram has $s_+$ adequate state.
Our choice of convention is dictated by the fact that that we want a $+$-smoothing
of the crossing in the diagram to correspond to the case when the edge is absent in the graph case;
compare Section 4.5 in \cite{H-P-R}.}. Similarly, $D$ is strongly $+$-adequate if it is strongly $s_-$-adequate,
and $D$ is strongly $-$-adequate if it is strongly $s_+$-adequate.
\end{enumerate}
\end{defn}

Figure \ref{WhG} shows a
diagram of a Whitehead link $D=D_{Wh}$, its $s_-$ smoothings (and $D_{s_-}$), $s_+$ smoothing
(and $D_{s_+}$), and their graphs $G(D)=G_{s_-}(D)$ and $(\bar D)=G_{s_+}(D)$. Notice that $D_{Wh}$ is strongly
$-$-adequate, i.e., $G(\bar D)=G_{s_+}(\bar D)=G_{s_-}(D)$ is a simple graph.

\subsection{The Kauffman bracket polynomial of graphs $[G]_{(\mu,A,B)}$}

When we translate the Kauffman bracket polynomial of diagrams $[D]_{(\mu,A,B)}$ to graphs
we obtain a version of the Tutte polynomial as explained below.

\begin{defn}\label{V.1.3}
The Kauffman bracket polynomial $[G]=[G]_{(\mu,A,B)}$ of the graph
$G$ ($[G] \in Z[\mu,A,B]$) is defined inductively by the following formulas\footnote{When comparing
with $<D>_{(\mu,A,B)}$ of Definition V.1.3 in \cite{Pr-1} we see that $[G]_{(\mu,A,B)}= \mu<G>_{(\mu,B,A)}$.
The Tait graph $G(D)$ of \cite{Pr-1} is our $G_{s_-}(D)$ graph. }:

\begin{enumerate}
\item[(1)]  $[U_n] = \mu^n$, where $U_n$ is the discrete graph on $n$ vertices.

\item [(2)] $[G] = A[G-e] + B[G//e]$ where $G//e = G/e$ if $e$ is not a loop, and if $e$ is a loop, then
 $G//e$ is defined to be  the graph obtained from $G-e$ by adding an isolated vertex.
\end{enumerate}
\end{defn}

The Kauffman bracket satisfies the following state sum formula:
\begin{lemma}\label{4:1.4}
\begin{equation*}
  [G] = \sum_{s \in 2^{E(G)}} \mu^{p_0([G:s])+p_1([G:s])} A^{|E(G) - s|} B^{|s|}
\end{equation*}
where the state $s$ is an arbitrary set of edges of $G$, including the empty
set, and $G-s$ denotes a graph obtained from $G$ by removing all edges contained in $s$.
\end{lemma}

The following formula expresses the relation between the Kauffman bracket and Tutte polynomial of
a graph $G$.
\begin{proposition}\label{Theorem 4:1.6}
The following identity holds
\begin{equation*}
  [G]_{(\mu,A,B)} =
\mu^{p_0 (G)} A^{p_1 (G)} B^{E(G)-p_1(G)} \chi (G;x,y)
\end{equation*}
where $x = \frac{B+\mu A}{B}$ and $ y = \frac{A+\mu B}{A}$.
\end{proposition}

\subsection{Khovanov homology via enhanced states chain complex}

Convenient way of defining Khovanov homology, as noticed by Viro \cite{Vi-1,Vi-2}, is
to consider enhanced Kauffman states\footnote{In Khovanov's original
approach every circle of a Kauffman state was decorated by a
$2$-dimensional module $A$ (with basis {\bf 1} and $X$)
with the additional structure  of Frobenius algebra
$A=\Z[X]/(X^2)$. Viro uses $-$ and $+$ in
place of {\bf 1} and $X$. }.
\begin{defn}\label{1.2}
An enhanced Kauffman state $S$ of an unoriented framed link diagram
$D$ is a Kauffman state $s$ with an additional assignment of $+$ or $-$ sign to each circle of $D_s.$
\end{defn}
The Kauffman bracket polynomial can be expressed as a sum of monomials using enhanced states which is
important in our approach to defining Khovanov homology. We have
$[D] = (-A^2-A^{-2})<D> = \sum_S (-1)^{\tau(S)}A^{\sigma(s)+2\tau(S)}$,
where $\tau(S)$ is the number of positive circles minus
the number of negative circles in the enhanced state $S$ (notice, that
$\tau(S) \equiv |D_s| \,\mod 2$).
\begin{defn}[Khovanov chain complex]\label{1.3}
Let ${\mathcal S}(D)$ denote the set of enhanced Kauffman states
of a diagram $D$, and let ${\mathcal S}_{i,j}(D)$ denote the set of enhanced
Kauffman states $S$ such that $\sigma(S) = i$ and $\sigma(S) +2\tau(S) = j$.
We call $i$ a homology grading and $j$ a Kauffman bracket grading.
\begin{enumerate}
\item[(i)]
The group ${\mathcal C}(D)$
(resp. ${\mathcal C}_{i,j}(D)$) is the free abelian group freely generated by
${\mathcal S}(D)$ (resp. ${\mathcal S}_{i,j}(D)$). ${\mathcal C}(D) =
\bigoplus_{i,j\in \Z} {\mathcal C}_{i,j}(D)$ is a bigraded free abelian
 group.
\item[(ii)] For a link diagram $D$ with ordered crossings, we define the
chain complex $({\mathcal C}(D),d)$ where $d=\{d_{i,j}\}$ and
the differential $d_{i,j}: {\mathcal C}_{i,j}(D) \to
{\mathcal C}_{i-2,j}(D)$ satisfies $d(S) = \sum_{S'} (-1)^{t(S:S')}[S:S'] S'$
with $S\in {\mathcal S}_{i,j}(D)$, $S'\in {\mathcal S}_{i-2,j}(D)$, and
 $[S:S']$
equal to $0$ or $1$. $[S:S']=1$ if and only if markers of $S$ and $S'$
differ exactly at one crossing, call it $v$, and all the circles
of $D_S$ and $D_{S'}$
not touching $v$ have the same sign\footnote{From our conditions
it follows that at the crossing $v$ the marker of $S$ is positive,
 the marker of $S'$ is negative, and
that $\tau(S') = \tau(S)+1$.}. Furthermore, $t(S:S')$ is the number of
negative markers assigned to crossings in $S$ bigger than $v$ in
the chosen ordering. %\textbf{First two rows of Table 8.1} show all possible
%types of pairs of enhanced states for which $[S:S']=1$. $S$ and $S'$ have
%different markers at the crossing $v$ and $\varepsilon=+$ or $-$.

\item[(iii)] The Khovanov homology of the diagram $D$ is defined to be
the homology of the chain complex $({\mathcal C}(D),d)$;
$H_{i,j}(D) = ker(d_{i,j})/d_{i+2,j}({\mathcal C}_{i+2,j}(D))$. The
Khovanov cohomology of the diagram $D$ is defined to be the cohomology
of the chain complex $({\mathcal C}(D),d)$.
\end{enumerate}
\end{defn}

Khovanov proved that his homology is a topological invariant.
$H_{i,j}(D)$ is preserved by the second and the third Reidemeister
moves and $H_{i+1,j+3}(r_{+1}(D)) = H_{i,j}(D)=H_{i-1,j-3}(r_{-1}(D))$.
Where $r_{+1}(\parbox{0.8cm}{\psfig{figure=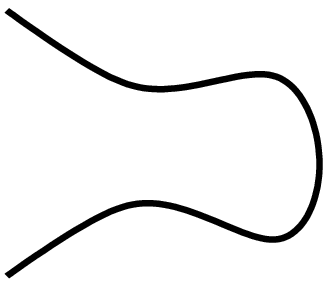,height=0.6cm}})=$
$(\parbox{1.0cm}{\psfig{figure=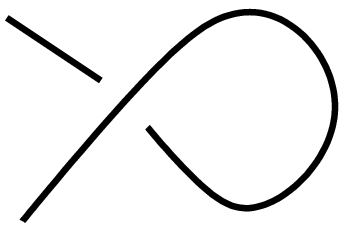,height=0.6cm}})$ and
$r_{-1}(\parbox{0.8cm}{\psfig{figure=kinksmooth.eps,height=0.6cm}})=$
$(\parbox{1.0cm}{\psfig{figure=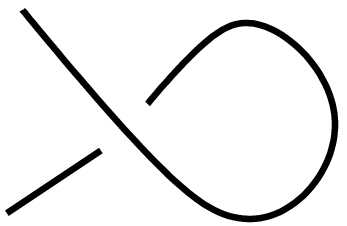,height=0.6cm}})$.

  With the notation we have introduced before, we can write
the formula for the Kauffman bracket polynomial of a link
diagram in the following form:
\begin{eqnarray*}
  [D] &=& \sum_j A^j (\sum_i (-1)^{\frac{j-i}{2}}\sum_{S\in S_{i,j}} 1)\\
      &=& \sum_j A^j (\sum_i (-1)^{\frac{j-i}{2}}dim C_{i,j})= \sum_j A^j \chi_{i,j}(C_{*,j}),
\end{eqnarray*}
 where $\chi_{i,j}(C_{*,j}) =
\displaystyle{\sum_{\begin{smallmatrix}
i:\,j\equiv i\\
 (mod\,2)
\end{smallmatrix}}}(-1)^{\frac{j-i}{2}}(dim C_{i,j})$
is a slightly adjusted Euler characteristic of the chain
complex $C_{*,j}$ ($j$ fixed). This explains that Khovanov homology
categorifies the Kauffman bracket polynomial, as well as the Jones polynomial\footnote{In
the narrow sense a categorification of a numerical or polynomial invariant
is a homology theory whose Euler characteristic or polynomial Euler
characteristic (the generating function of Euler characteristics), respectively are equal to the
the invariant we have started with. We quote M.~Khovanov \cite{Kh-1}:
 ``A speculative question now comes to mind: quantum invariants of knots
and $3$-manifolds tend to have good integrality properties. What if
these invariants can be interpreted as Euler characteristics of some
homology theories of $3$-manifolds?".}.

%\section{Comultiplication free chain complex of graphs and their ``first comultiplication" version}

\subsection{Khovanov type functor on the category of graphs}

The chromatic graph cohomology was introduced in \cite{HR}, as a comultiplication free version of the Khovanov cohomology of
alternating links, where alternating link diagrams are translated to plane graphs (Tait graphs). Moreover, this homology theory is a categorification of the chromatic polynomial of a graph.

The chromatic polynomial of a graph counts the number of its proper vertex colorings using no more than a given number of colors,
 so that adjacent vertices have different colors.
The analogy with the Khovanov homology construction is almost complete:
%(except of a shift of grading and lack of comultiplication)
instead of Kauffman states we use subgraphs $[G:s]$, containing all vertices in $G$ and
$i$ edges in $s \subseteq E$. Analogously to the enhanced Kauffman states, we label connected
components of a graph $[G:s]$ by either $1$ or $x$, the generators of algebra $\A_2=\Z/(x^2)$.
 The number of circles in the Kauffman state $|D_s|$ corresponds to the number of connected
components $k(s)$ of the graph $[G:s]$ containing all vertices in $G$ and $i$ edges
in $s \subseteq E$. Now, we consider the following state sum formula for the chromatic polynomial:

 \begin{eqnarray*}\label{CGstatesum}
   \chi_G(\lambda)&=&\displaystyle{\sum_{i\geq 0} (-1)^i \sum_{s \subseteq E, |s|=i} \lambda^{k(s)} }\\
               &=& \displaystyle{\sum_{i,j \geq 0} (-1)^i  \lambda^{j} \sharp \{s \subseteq E| |s|=i, k'([G:s])=j \} }.
 \end{eqnarray*}
where $k'([G:s])$ denotes the number of components of $[G:s]$ labeled by $x$.
 Cochain groups are spanned by all subgraphs $[G:s]$, with each of $k([G:s])=j$ components labeled by either $1$ or $x$, with exactly $j$ components labeled by $x$.

The differential takes an enhanced state, i.e. labeled subgraph $[G:s]$ into a state $[G:s]'$ where $S'=S \cup {e}$, $e \in E \setminus S$.
 Components of the obtained state are labeled  using multiplication  rules in $\A_2$: if two different components are merged new label is a product of the old ones, otherwise, if a cycle is closed it is just an identity.

\begin{defn}\label{Definition 2.3}
 Given commutative algebra $\A_2$ we define the chromatic cochain complex and
chromatic cohomology of a graph $G=(V(G),E(G))$ with $V(G)$ set of vertices and $E(G)$ the set of edges
 in the following way:
\begin{enumerate}
\item[(i)] The cochain group $C^i(G) = \displaystyle{\bigoplus_{|s|=i, s \subset E(G)} C^i_s(G)},$ with
$C^i_s(G) = \A^{k(s)}$ where $k(s)$ denotes the number of components of the graph $[G:s]$ which is the
subgraph of $G$ containing all vertices of $G$ and edges in $s.$

Assume that the edges of $G$ are ordered\footnote{The chromatic graph cohomology is independent on
the ordering of edges, however, the ordering is required for defining the boundary map.}. For a given state $s$ and
the edge $e \in E \setminus s$, let
$t(s,e)$ equal to the number of edges in $s$ less than $e$ in the chosen ordering.
The cochain map $d^i: C^i(G) \to C^{i+1}(G)$ is a sum
$$d^i = \sum_{e \notin s} (-1)^{t(s,e)} d_e^i$$
where the map $d_e^i$ depends on whether $e$ connects different components of
$[G:s]$ or it connects vertices in the same component of $[G:s].$
In the last case we assume $d_e^i$ to be the identity (or alternatively zero
map \cite{HR}).
If $e$ connects different components of $[G:s]$, say $i-th$ and $j$-th, $i<j$ then:
$d_e^i(a_1,a_2,\ldots, a_{k(s)-1}) = (a_1,a_2,\ldots, a_ia_j,\ldots, a_{j-1},a_{j+1},\ldots, a_{k(s)-1})$.

\item[(ii)] We define the chromatic cohomology denoted by $H^*(G)$ as the cohomology of the
above chromatic cochain complex.
\end{enumerate}
\end{defn}

The chromatic graph cohomology of a graph with a loop is always equal to zero (the chromatic polynomial as well)
since the chain complex can be presented as a cone of two isomorphic chain complexes.

In this setting it is easier to work with the chain complex, similar to the classical homology theories.
Therefore we perform concrete calculations in the chromatic graph homology setting and then use the universal
coefficient theorem (see Proposition \ref{Proposition 2.9}) to translate the results to the chromatic graph cohomology.

\begin{prop}\label{Proposition 2.9}
If the homology groups $H_n$ and $H_{n-1}$ of a chain complex
$C$ of free abelian groups are finitely generated then
$$H^n(C;\Z) = H_n(C;\Z)/ \tor(H_n(C;\Z)) \oplus \tor(H_{n-1}(C;\Z)).$$
\end{prop}

In particular, we use the following identities:
\begin{itemize}
  \item $H^{0,v-1}_{\A_2}(G) =
H_{0,v-1}^{\A_m}(G)/tor(H_{0,v-1}^{\A_2}(G))$,
  \item $H^{1,v-1}_{\A_2}(G) = H_{1,v-1}^{\A_2}(G) \oplus
tor(H_{0,v-1}^{\A_2}(G))$ and
  \item $H_{1,v-1}^{\A_2}(G) = ker(C_{1,v-1}^{\A_2}(G) \to
C_{0,v-1}^{\A_2}(G))$ is a free abelian group.
\end{itemize}

The chromatic graph cohomology of a graph $G=(V(G),E(G))$
over algebra $\A_2$ is equivalent to the homology with chain groups defined as in the
standard graph homology and the boundary map defined by
$\partial(e)=\partial(\overrightarrow{V_1 V_2})=V_1+V_2$ where $e=(V_1,V_2) \in E$.
As a corollary we get the following lemma from \cite{PPS}:
\begin{prop}
Let G be a connected simple graph, then:
\begin{equation}
  H_{0,v-1}^{\A_2}(G)=\left\{%
\begin{array}{ll}
    \mathbb{Z}, & \hbox{if G is a bipartite  graph;} \\
    \mathbb{Z}_2, & \hbox{if G has an odd cycle.} \\
\end{array}%
\right.
\end{equation}
\end{prop}

%**********************************************************************************************

Consider the category of finite graphs in which finite graphs are objects and
$Mor (H,G)$) are graph embeddings between $H$ and $G$ (for simplicity we assume that for a morphism
$f:H\to G$, $f(H)$ is a spanning graph, that is, it contains all vertices of $G$).
To every graph we associate its chain complex $\{C_{i,j}(G)\}$, moreover any $f:H \ra G$ induces
the chain map $f_{\#}: \{C_{i,j}(H)\} \to \{C_{i,j}(G)\}$. We obtain in such a way
a functor from the category of finite graphs to the category of graded  chain complexes, and
further to the category of bi-graded groups $\{H_{i,j}(G)\}$.
In a standard way we consider the morphism $\alpha$ of $H $ in $ G$ and related
short exact sequence of chain complexes
\begin{equation*}
  0\to C_{i,j}(H) \to C_{i,j}(G) \to C_{i,j}(G,H)\to 0
\end{equation*} where $C_{i,j}(G,H)=C_{i,j}(G)/C_{i,j}(H)$.
Finally, we obtain the related long exact sequence of homology:\\
\parbox{0.5cm}{%
\begin{equation}
  \label{MainLemmaExactSeqGT} \\ \,{\textbf{}}
  \parbox{15cm}{\vspace{-0.8cm}
\begin{eqnarray*}
 &\ldots& \to H_{i,j}(H) \to H_{i,j}(G) \to H_{i,j}(G,H)\to  H_{i-1,j}(H)\to \ldots\\
  &\ldots& \to  H_{1,j}(G) \to H_{1,j}(G,H) \to H_{0,j}(H) \to H_{0,j}(G) \to H_{0,j}(G,H)\to 0
\end{eqnarray*}}
\end{equation}}

\begin{prop}\label{GraphSubgraph}
Let $T$ be a spanning tree of a connected graph $G$, then
\begin{enumerate}
\item[(i)]
$ H_{*,*}(T)= H_{0,v-1}(T)\oplus H_{0,v}(T) =\Z[0]\{v-1\}\oplus \Z[0]\{v\}$.
%We use notation in which, say $Z[i]{j}$ denotes $Z$ with homological grading $i$ and chromatic
%grading $j$.
\item[(ii)] Homology is supported on two diagonals, $H_{i,j}(G)=0$ for $i+j\neq v,v-1$, and
the torsion is trivial except possibly for $i+j=v-1$: $\tor H_{i,j}(G)=0$ for $i+j\neq v-1$.
\item[(iii)]  $H_{i,j}(G)= H_{i,j}(G,T)$ if $i>1$ or $i=1$ and $j\neq v-1.$ In particular, $H_{1,v-2}(G)= H_{1,v-2}(G,T).$
\end{enumerate}
\end{prop}
\begin{proof} \begin{itemize}
 \item[(i)] Let $G_1*G_2$ denote one--vertex product of graphs and  $K_n$ the complete graph  on
$n$ vertices. Adding a loop loop $K_1$ to a graph $G$ (see \cite{HR}) on the level of knots,
corresponds to applying first Reidemeister move to the knot.
Part (i) is the special case of the fact that Khovanov homology of links is shifted under the first
Reidemeister move $H_{i,j}(G*K_1)= H_{i,j+1}(G)$.
   \item[(ii)]Part (ii) reflects the fact that Khovanov homology of alternating links lies
on two adjacent diagonals \cite{Lee}. The proof uses the long exact sequence of homology
using smoothings in a link case and deleting--contracting in the graph case (see \cite{HR,A-P}).
 \item[(iii)] The third part follows from parts (i) and (ii) by applying the long exact sequence of the pair $(G,T)$.
\end{itemize}
 \end{proof}
%*******************************************************************************************

\subsection{Correspondence between Khovanov and chromatic graph homology} \label{KhGr}

Based on \cite{H-P-R, Pr-2} we state the following Proposition \ref{GraphKhovanov} that establishes
the relation between graph cohomology and classical Khovanov
homology of alternating links (in Viro \cite{Vi-1} notation). Proposition \ref{GraphKhovanov} is
generalized in Proposition  \ref{GraphKhovanovModified}.

\begin{prop}\label{GraphKhovanov}
Let $D$ be the diagram of an unoriented framed alternating link
and let $G= G_{s_-}(D)$\footnote{In the case of oriented alternating links $G_{s_-}(D)$ and $G_{s_+}(D)$ are
Tait graph, i.e. a graph obtained from the checkerboard coloring of the projection plane.}. Let $\ell$
denote the girth of $G$, that is, the length of the shortest cycle
in $G$. For all $i<\ell-1$, we have: $$ H^{i,j}_{\A_2}(G) \cong
H_{a,b}(D)$$ where $a=E(G)-2i$, $b=E(G)-2v(G)+4j$ and $H_{a,b}(D)$
are the Khovanov homology groups of the unoriented framed link
defined by $D$, as explained in Definition 2.7 based on \cite{Vi-1}.

Furthermore, $tor(H^{i,j}_{\A_2}(G)) = tor(H_{a,b}(D))$ for
$i=\ell-1$.
\end{prop}

\begin{theorem} \cite{PPS}\label{Theorem 3.1}
Let $G$ be a simple graph then
\begin{enumerate}
\item[(1)] $H^{0,v-1}_{\A_2}(G)= \Z^{p_0^{bi}}$, where $p_0^{bi}$
is the number of bipartite components of $G$.
\item[(2)]
$H^{1,v-1}_{\A_2}(G)= \Z^{p_1- (p_0 -p_0^{bi})} \oplus \Z_2^{p_0
-p_0^{bi}}$, where $p_0$ is the number of components of $G$ and
$p_1= rank (H_1(G,\Z))= |E|-v+p_0$ is the cyclomatic number of
$G$.
\end{enumerate}
\end{theorem}

%----------------------------------------------------------------------------------------------
\section{The Main Lemma and chromatic graph homology $H_{1,v-2}$}
\label{DeeperA2}
%----------------------------------------------------------------------------------------------

Next we compute $H_{1,v-2}(G)$ for any connected graph $G$, hence $H^{2,v-2}(G)$ for any graph $G$ and, eventually,
$H_{n-4,n+|D_{s_+}|-8}(D)$ for a corresponding $+$-adequate link diagram.

\begin{lemma}[Main Lemma]\label{Main Lemma}\ \\
If $G$ is a connected simple graph, i.e., of girth $l(G) \geq 3$, then:
\begin{enumerate}
\item[(i)] $H_{1,v-2}(G)= \Z_2^{p_1(G)}$, if G is bipartite.
\item[(ii)] $H_{1,v-2}(G)= \Z_2^{p_1(G)-1} \oplus \Z$, if G has an odd cycle.
\end{enumerate}
\end{lemma}
\begin{proof}
Since $H_{1,v-2}(G)=H_{1,v-2}(G,T)$, for any spanning tree $T$ of $G$, Proposition \ref{GraphSubgraph}, we focus on computing $H_{1,v-2}(G,T)$.
We assume that both edges and vertices are ordered, although results do not depend on it. To make the proof more comprehensible we introduce the following notation. Let $\rho(v,w)$ denote the distance between vertices $v, w \in V(T)$, equal to the length of the shortest path connecting them in $T$.
If $(\partial_0(e),\partial_1(e))$ denotes endpoints of the $e$ edge in $G$, we use the
short notation $\rho(e_i)=\rho(\partial_0(e_i),\partial_1(e_i)).$ In particular,
$\rho(e_i)$ is odd if $e$ closes an even cycle in $T \cup e_i$, and $\rho(e_i)$ is even  if $e$ closes an
odd cycle in $T\cup e_i$. We also use $\rho(e_i,e_j)$ to denote the distance between $e_i$ and $e_j$ in $T \cup e_i \cup e_j$,
or equivalently, as the minimal distance between endpoints of $e_i$ and endpoints of $e_j$ in $T$.

Let $e_1,...,e_{p_1}$ be the edges in $E(G \setminus T)$ where $p_1(G)$ is the cyclomatic number of $G$,
$p_1(G)= |E(G)|-|E(T)| =|E(G)|- |V(G)|+1$.

 The chain group $C_{1,v-2}(G,T)$ is freely generated by enhanced
states $(e_i,v_j)$ where the component of the graph $[G:e_i]$ containing the vertex $v_j$ has
the label $1$ (all other labels are $x$). If the vertex $v_j$ is the endpoint of $e_i$ we use short notation  $(e_i,1)$ for an enhanced state
$(e_i,\partial_0(e_i))=(e_i,\partial_1(e_i))$.

Notice that $H_{0,v-2}(G,T)=0=C_{0,v-2}(G,T)$ since $C_{0,v-2}(G)=C_{0,v-2}(T)$. Therefore, $ker (d: C_{1,v-2}(G,T) \ra C_{0,v-2}(G,T))=C_{1,v-2}(G,T)$ so:
\begin{equation*}
  H_{1,v-2}(G,T)= C_{1,v-2}(G,T)/d(C_{2,v-2}(G,T)).
\end{equation*}

The chain group $C_{2,v-2}(G,T)$ has two types of free generators (enhanced states):
\begin{itemize}
  \item pairs $(e_i,e)$ where $e\in E(T)$ generating the subgroup of $C_{2,v-2}(G,T)$ denoted by $C'$ and
  \item pairs $(e_i,e_j)$ generating the subgroup  $C''.$
\end{itemize}

  Let us first compute $C_{1,v-2}(G,T)/d(C')$. For any edge $e \in T$,
$$d (e_i,e)= \pm((e_i,\partial_1(e)) + (e_i,\partial_0(e)))$$
yields the following relation in homology $(e_i,\partial_1(e))=-(e_i,\partial_0(e)).$
Hence, we eliminate all generators of $C_{1,v-2}(G,T)$ except pairs $(e_i,\partial_0(e_i))$, satisfying  relations
\begin{eqnarray*}
 (e_i,\partial_1(e_i))&=&(-1)^{\rho(\partial_0(e_i),\partial_1(e_i))}(e_i,\partial_0(e_i)) {\rm \,  that \, is }\\
 (e_i,1)&=&(-1)^{\rho(e_i)}(e_i,1).
\end{eqnarray*}
 Thus $C_{1,v-2}(G,T)/d(C')= \Z_2^{k_{odd}}\oplus \Z^{p_1-k_{odd}}$,
where $k_{odd}$ is the number of edges $e_i$ with $\rho(e_i)$ odd.

%In the case of $G$ being bipartite graph only the first case happen so let us start from this:\\
Next, we compute $(C_{1,v-2}(G,T)/d(C'))/d(C'')$. For an enhanced state $(e_i,e_j)$ we have:
\begin{equation}\label{Equation}
  d(e_i,e_j)=\pm((e_i,\partial_0(e_j))+(e_i,\partial_1(e_j)) -
(e_j,\partial_0(e_i)) - (e_j,\partial_1(e_i)))
\end{equation}
  The relation in $C_{1,v-2}(G,T)/d(C')$ corresponding to Equation \ref{Equation}
can be written as
\begin{equation}\label{RelEquation}
  (e_i,1)(1+(-1)^{\rho(e_j)}) = \varepsilon ((e_j,1)(1+(-1)^{\rho(e_i)}),
\end{equation}
 where $\varepsilon=\pm 1$, or more precisely $\varepsilon=(-1)^{\rho(\partial_0(e_i),\rho(\partial_0(e_j))}$.
We analyze this relation in more detail, based on the types of enhanced states generating $C''$.
Depending on the parity of $\rho(e_i,e_j)$ we get three different types of generators of $C''$:
\begin{enumerate}
  \item $(e_i,e_j)$ such that both $\rho(e_i)$ and $\rho(e_j)$  are odd generate the subgroup  $C''_{odd}$
  \item $(e_i,e_j)$ where exactly one of $\rho(e_i)$ and $\rho(e_j)$  is odd generate the subgroup  $C''_{mixed}$
  \item $(e_i,e_j)$ such that both $\rho(e_i)$ and $\rho(e_j)$  are even generate the subgroup  $C''_{even}$
\end{enumerate}

In the first case $C''_{odd}$,  both sides of Equation(\ref{Equation}) are equal to $0$, so there are no new
relations in  $C_{1,v-2}(G,T)/d(C')$.
If a graph $G$ is bipartite $\rho(e_i)$ is always odd, $C_{2,v-2}(G,T)$ is generated by $C'$ and  $C''_{odd},$ so
Part (i) of Main Lemma is proven.

In the second case, Equation (\ref{RelEquation}) reduces to
the equation $2(e_i,1)=0$ which already holds in $C_{1,v-2}(G,T)/d(C')$.

Finally, consider the third case when  $2(e_i,1)=\varepsilon 2(e_j,1)$, or
more precisely \begin{equation}
  2((e_i,1)-(-1)^{\rho(e_i,e_j)}(e_j,1)))=0. \label{relML}
\end{equation}
%Notice that
%$$\rho(e_i,e_j)\equiv \rho(\partial_0(e_i)\partial_0(e_j)) \equiv \rho(\partial_1(e_i),\partial_1(e_j))(\mod 2)$$
%since both $\rho(e_i)$ and $\rho(e_j)$ are even.

To conclude the proof of part (ii), let $e_1,..,e_k$ be edges of $G\setminus T$ with odd $\rho(e_i),$
and $e_{k+1},...,e_{p_1}$ be the remaining edges, i.e. edges with $\rho(e_i)$ even.
Graph $H$ obtained by adding $e_1,..,e_k$ to the tree $T$ is a bipartite graph, so we know that
 $$H_{1,v-2}(H,T)= C_{1,v-2}(H,T)/d(C')= \{\{(e_i,1)\}_{i=1}^k\ | \ 2(e_i,1)=0\}=\Z_2^k,$$
and that $H_{1,v-2}(G,T)= H_{1,v-2}(H,T)/C''_{even}.$
Observe now that for $e_i,e_j$ in $E(G) \setminus E(H)$ the relation (\ref{relML}) follows from relations $2((e_i,1)-(-1)^{\rho(e_i,e_{k+1})}(e_{k+1},1))=0$ and
$2((e_j,1)-(-1)^{\rho(e_j,e_{k+1})}(e_{k+1},1))=0$ since $\rho(e_i,e_j)\equiv \rho(e_i,e_{k+1}) + \rho(e_j,e_{k+1}) (\mod 2)$.
Hence, $H_{1,v-2}(G,T)$ is generated by\\
 $(e_1,,...,e_k, e_{k+1}, (e_{k+2}-(-1)^{\rho(e_{k+2},e_{k+1})}(e_{k+1},1)),\ldots, (e_{p_1}-(-1)^{\rho(e_{p_1},e_{k+1})}(e_{k+1},1))$,
where $e_{k+1}$ is an infinite cyclic element and all other generators have order $2$.
The proof of Main Lemma is completed.
\end{proof}
As a corollary we get the following main result.

\begin{theorem}\label{PapCorollary 4.3}
\begin{enumerate} If $G$ is a connected simple graph containing $t_3$ triangles, then:
\item [(1)]
$H^{}_{1,v-2}(G)=\left\{%
\begin{array}{ll}
    \Z_2^{p_1(G)}, & \hbox{if G is bipartite;} \\
    \Z_2^{p_1(G)-1} \oplus \Z, & \hbox{if G has an odd cycle.} \\
\end{array}%
\right.$
\item[(2)]
$H^{}_{2,v-2}(G)=\left\{%
\begin{array}{ll}
    \Z^{{p_1 \choose 2}-t_3}, & \hbox{if G is bipartite;} \\
    \Z^{{p_1 \choose 2}-t_3+1}, & \hbox{if G has an odd cycle.} \\
\end{array}%
\right.$
\item[(3)]
$ H^{1,v-2}_{}(G) =\left\{%
\begin{array}{ll}
    0, & \hbox{if G is bipartite;} \\
    \Z, & \hbox{if G has an odd cycle.} \\
\end{array}%
\right.$
\item[(4)] $ H^{2,v-2}_{}(G) =\left\{%
\begin{array}{ll}
Z_2^{p_1} \oplus Z^{{p_1 \choose 2}-t_3}, & \hbox{G is bipartite;} \\
    Z_2^{p_1 - 1} \oplus Z^{{p_1 \choose 2} +1-t_3}, & \hbox{if G
has an odd cycle.} \\
\end{array}%
\right.$

\end{enumerate}
\end{theorem}
\begin{proof}
(1) follows from Main Lemma and Proposition \ref{GraphSubgraph}(iii).\\
(2) Using Euler characteristic of chromatic graph cohomology in degree $j=v-2$ we get:
$$\rank H^{}_{2,v-2}(G)-\rank H^{}_{1,v-2}(G)=a_{v-2}$$
where $a_{v-2}$ denotes the coefficient of $q^{v-2}$ in the chromatic
polynomial equal to\footnote{To put our
calculation in a general combinatorial context we note that we have the
following identity which we use here only for $i=2$ and in the full
generality in a sequel paper:\\
 $\sum_{i=0}^E (-1)^i{E \choose i}\lambda^{v-i} =
\lambda^{v-E}(\lambda-1)^E \stackrel{\lambda=q+1}{=} (q+1)^{v-E}q^E = q^v (1+q^{-1})^{-(E-v)} =
\sum_{i=0}^{\infty} (-1)^i {(E -v+ 1 )+ i-2 \choose i}q^{v-i} =
\sum_{i=0}^{\infty} (-1)^i {p_1+ i-2 \choose i}q^{v-i}.$}:

\begin{equation}
a_{v-2}={E \choose 2}-t_3-E(v-1)+{v \choose 2}={p_1 \choose 2}-t_3.
\end{equation}

Parts (3) and (4) follow directly from (1) and (2) by applying universal coefficient theorem
($H^{2,v-2}(G)= \free(H_{2,v-2}(G))\oplus \tor H_{1,v-2}(G)$,
and $H^{1,v-2}(G)= \free(H_{1,v-2}(G)$).
\end{proof}

The restriction to connected graphs was made only for simplicity.  K\"unnet formula is sufficient
for recovering homology of the graph from the homology of the connected components (compare with \cite{HR}).
In fact, when computing the homology of a disjoint sum of graphs, $H^{**}(G_1\sqcup G)$,
we can sometimes ignore the torsion part of the formula.

\begin{cor}\label{PapCorollary 4.4}
Let $G$, $G_1$ and $G_2$ denote arbitrary graphs, $G^{bi}$  all
bipartite components of $G$, and $G^{nbi} = G-G^{bi}$ the remaining components of the graph $G$.
\begin{enumerate}
\item[(1)] $H^{i,v-i}_{\A_2}(G_1 \sqcup G_2) = H^{**}_{\A_2}(G_1) \otimes H^{**}_{\A_2}(G_2) $ in bidegree $(i,v-i) .$
%(MUST be true for any number of components).

\item [(2)] If $G^{bi}$ and $G^{nbi}$ are simple graphs then:
\begin{eqnarray}\label{BiNBiH2}
   H^{2,v-2}_{\A_2}(G^{bi}) &=& \Z_2^{p_1(G^{bi})} \oplus  \Z^{{p_1(G) \choose 2}} \\
   H^{2,v-2}_{\A_2}(G^{nbi}) &=& \Z_2^{p_0(G^{nbi})p_1(G^{nbi})- {p_0(G^{nbi})+1 \choose 2}} \oplus  \Z^{\alpha}
\end{eqnarray}
where $\alpha= {p_1(G^{nbi})+1 \choose 2}- (p_0(G^{nbi})-1)p_1(G^{nbi})- {p_0(G^{nbi})+1 \choose 2}  .$
\item[(3)] If $G$ is simple graph with $p_0(G)$ components then
%$$tor H^{2,v-2}{}(G) = \Z_2^{p_1(G)-p_0(G)+p_0(G^{bi}) + p_0(G^{nbi})p_1(G)- p_1(G^{nbi})} $$
%$$tor H^{2,v-2}_{\A_2}(G) = \Z_2^{p_1(G^{bi})-p_0(G^{nbi})+ p_0(G^{nbi})p_1(G)}.$$
\begin{equation}\label{GH2}
  \tor H^{2,v-2}_{\A_2}(G) = \Z_2^{p_1(G^{bi})+p_0(G^{nbi})p_1(G)+ {p_0(G^{nbi})+1 \choose 2}}.
\end{equation}

\item[(4)] If $G$ is simple graph with $p_0(G)$ components then
\begin{equation*}\label{GH2free}
  \rank H^{2,v-2}_{\A_2}(G) = {{p_1(G)+1 \choose 2}- \dim \tor H^{2,v-2}_{\A_2}(G) -t_3(G) }.
\end{equation*}
\end{enumerate}
\end{cor}

\begin{proof}
\begin{enumerate}
  \item [(1)]K\"unneth formula yields the following formula for chromatic graph cohomology over $\A_2$
\begin{equation}
 H^{i,j}(G_1 \sqcup G_2) = \bigg(\mathop{\bigoplus_{p+q=i}}_{s+t=j}H^{p,s}(G_1)
 \otimes H^{q,t}(G_2)) \oplus \mathop{(\bigoplus_{p+q=i+1}}_{s+t=j}H^{p,s}(G_1) *_{Tor} H^{q,t}(G_2)) \bigg)
\end{equation}
thus it suffices to show that
\begin{equation*}
 \mathop{\bigoplus_{p+q=i+1}}_{s+t=j}H^{p,s}_{\A_2}(G_1) *_{Tor} H^{q,t}_{\A_2}(G_2)=0
\end{equation*}
in bidegrees $(i,j)$ satisfying $i+j=v(G_1\sqcup G_2)$.

If $G$ is a connected graph
\begin{itemize}
   \item homology is supported in bidegrees $(i,j)$ satisfying $v(G)-1 \leq i+j \leq v(G)$
   \item torsion is supported in $(i,j)$ such that $i+j= v(G)$.
 \end{itemize}

By induction on the number of components and  K\"unneth formula we get a well known fact (compare with \cite{A-P,H-P-R})
for an arbitrary graph $G$
\begin{itemize}
   \item homology is supported in bidegrees $(i,j)$ such that $v(G)-p_0(G) \leq i+j \leq v(G)$
      \item torsion is supported $(i,j)$ such that $v(G) -p_0(G)+1 \leq i+j \leq v(G)$.
 \end{itemize}

From the second inequality and the K\"unneth formula
we are interested only in bidegrees satisfying $p+q+r+s= v(G_1\sqcup G_2)+1 = v(G_1)+ v(G_2)+1.$
However, this implies that either $p+q \geq v(G_1)$ or $s+t \geq v(G_2)$, which contradicts the previous
observation.  Hence, \begin{equation*}
  \mathop{\bigoplus_{p+q=i+1}}_{s+t=j}H^{p,s}(G_1) *_{Tor} H^{q,t}(G_2)
\end{equation*} is trivial.
% The whole second part can be always ignored, not only the torsion one, by $\leq v(G)$
  \item [(2)] According to part (1) we have:
\begin{eqnarray*}
  H^{2,v(G_1\sqcup G_2)-2}_{\A_2}(G_1\sqcup G_2)&=& H^{2,v(G_1)-2}_{\A_2}(G_1) \oplus H^{2,v(G_2)-2}_{\A_2}(G_2) \oplus \\
                                                & \oplus & (H^{1,v(G_1)-1}_{\A_2}(G_1) \otimes H^{1,v(G_2)-1}_{\A_2}(G_2)).
\end{eqnarray*}

We apply this formula inductively, using results from Theorem \ref{PapCorollary 4.3}(iv) and Theorem \ref{Theorem 3.1}, to obtain
formulas (\ref{BiNBiH2}) and (\ref{GH2}). The intermediate step is computing
$H^{2,v-2}_{\A_2}(G^{bi})$ assuming that $G^{bi}=G_1^{bi} \sqcup G_2^{bi} \sqcup \ldots \sqcup G^{bi}_{p_0(G^{bi})}$:

\begin{equation*}
  H^{2,v-2}_{\A_2}(G^{bi})= \Z_2^{p_1(G^{bi})} \oplus \Z^{  {p_1(G^{bi}) \choose 2}}
\end{equation*}
using the identity ${p_1(G_1)\choose 2}+...+ {p_1(G_{p_0(G)})\choose 2}+ \sum_{i<j}p_1(G_i)p_1(G_j) = {p_1(G)\choose 2}$.

Similarly, assuming that $G^{nbi}=G_1^{nbi} \sqcup G_2^{nbi} \sqcup \ldots \sqcup G^{nbi}_{p_0(G^{nbi})}$:

\begin{equation*}
  H^{2,v-2}_{\A_2}(G^{nbi})= \Z_2^{p_0(G^{nbi})p_1(G^{nbi})+ {p_0(G^{nbi})+1 \choose 2}} \oplus \Z^{ \alpha}
\end{equation*}
where $\alpha= {p_1(G^{nbi})+1 \choose 2}- (p_0(G^{nbi})-1)p_1(G^{nbi})- {p_0(G^{nbi})+1 \choose 2} -t_3(G^{nbi}).$

  \item [(3)-(4)] Proof of part (3) follows from results in (2) and to complete the proof of part (4) notice that that for any simple graph $G$
\begin{equation*}
  rank (free (H^{2,v(G)-2}(G))) + dim (tor( H^{2,v(G)-2}(G))))= {p_1(G)+1 \choose 2}- t_3(G).
\end{equation*}

\end{enumerate}

\end{proof}

\begin{cor}\label{CompleteGraphH2v-2}
Let $K_n$ denote the complete graph with $n$ vertices. Then
$$H^{2,v-2}_{\A_2}(K_n)=\Z^{\frac{n(n-3)}{2}} \oplus \Z^{3 {n \choose 4}+1-{n \choose 3}}$$

\end{cor}

\begin{cor}\label{WheelsH2v-2}
Let $W_n$ denote the wheel graph with $n+1$ vertices, i.e. the cone over an $n$-gon. Then
$$H^{2,v-2}_{\A_2}(W_n)=\Z_2^{n-2} \oplus \Z^{ {n-1 \choose 2}-n+1}$$

\end{cor}

%-----------------------------------------------------------------------------------------------
\section{Adding comultiplication to chromatic graph cohomology}
\label{A2comultiplication}
%------------------------------------------------------------------------------------------------

In order to make further use of correspondence between Khovanov and chromatic graph cohomology
described in Subsection  \ref{KhGr} and \cite{A-P,H-P-R,Pr-2,PPS}, we adjust the original definition
by incorporating comultiplication in the differential.
This modification extends the correspondence between Khovanov homology and chromatic graph cohomology to additional
homological grading. In particular, this definition enables computing torsion in Khovanov homology in bidegree $(n-4, n+2|D_{s_+}|-8)$.

\begin{figure}[h]\begin{center}
   \includegraphics[width=12cm]{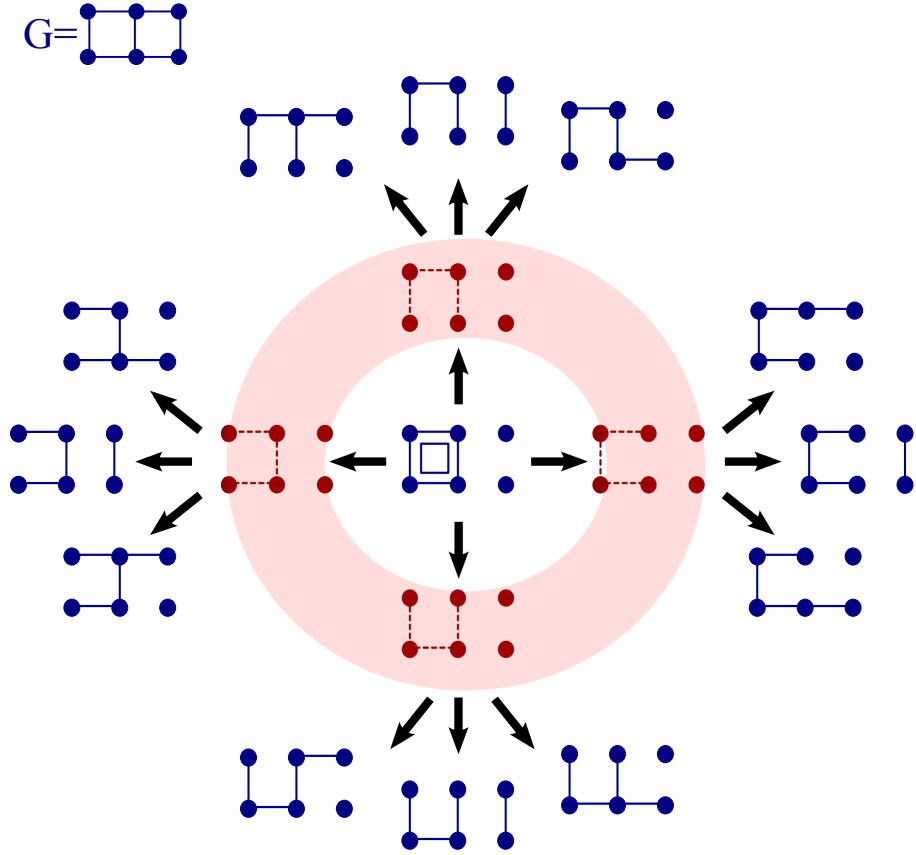}
   \caption{Examples of generators of ${}^{\Delta}C_3(G)$ (shown in red, dashed, in a shaded annulus) and ${}^{\Delta}C_4(G)$ (shown blue) for graph $G$ in the top left corner. }\label{ex}
\end{center}
\end{figure}

First, the chain complex is adjusted so that it can accommodate
comultiplication. The original cochain groups contain a copy of
algebra for each connected component in the graph $[G:s]$, see Definition \ref{Definition 2.3}.
Cochain groups ${}^{\Delta}C^i(G)$ stay the same for $i<l(G)$, and trivial for $i>l(G)$.
Modified cochain groups ${}^{\Delta}C^i(G)$ will contain a tensor product $\A_2 \otimes \A_2$ instead of a single copy of algebra $\A_2$
for each state containing a closed cycle. Pictorially, the component containing a
closed cycle is decorated by basis elements of a tensor product $\A_2 \otimes \A_2$, see Figure \ref{ex}. This description is formalized in the following definition.

\begin{defn}\label{ModifiedChGr}
For a given graph $G$ of girth $l$, let ${}^{\Delta}C^{i,*}(G)$
denote modified chromatic chain groups defined in the following way:
\begin{enumerate}
    \item ${}^{\Delta}C^{i,*}(G)\cong C^{i,*}(G)$
    for $i < l$,
    \item ${}^{\Delta}C^{i,*}(G)\cong \bigoplus_{|s|=i} \A_2^{\otimes (p_0([G:s]) + p_1([G:s]))}$ for $i = l$,
    \item ${}^{\Delta}C^{i,*}(G)=0$  for $i > l$.
\end{enumerate}
\end{defn}

Next, we introduce comultiplication in the differential the first time when adding an edge preserves the number
 of connected components of a graph, i.e. when it is closing shortest cycles.
 Let  ${}^{\Delta}d_{s,e}$ denote
 the modified differential. If $p_1([G:s]) = p_1([G:s \cup e])= 0$, the differential
 stays the same ${}^{\Delta}d_{s,e}=d_{s,e}.$
\iffalse
 \begin{figure}[h]\label{FNewComplexH2}
  \begin{center}
\scalebox{.7}{\includegraphics{H2NewComplex.eps}}
\caption{Chain complex generators for the modified chromatic graph cohomology of a graph $G.$}
\end{center}
\end{figure}
\fi
If the edge $e$ we are adding is an internal edge  of $[G:s]$ (i.e.
$1=p_1([G:s\cup e]) =p_1([G:s])+1$), the differential is determined by comultiplication in algebra $\A_2$,
given by $\triangle(1)=(1 \otimes x) + (x \otimes 1)$ and $\triangle(x)= x \otimes x $.

We have all necessary ingredients to define the differential.
\begin{defn} The differential map
${}^{\Delta}d^{i}(G): {}^{\Delta}C^i(G)\rightarrow
{}^{\Delta}C^{i+1}(G)$ is defined by
$${}^{\Delta}d^{i}[G:s]=\sum_{e \in E(G)\setminus s}(-1)^{t(s,e)}d_e([G:s])$$
 where $[G:s] \in {}^{\Delta}C^i(G)$ and $t(s,e)=|\{e' \in s| e' <e\}|$ for all
$i<l(G)$. Let $c_1, \ldots c_k$ denote the components of the state $[G:s].$
  The definition of the map $d_e$ varies depending on whether the edge e is connecting
two different components of $[G:s]$, say
$c_m$ and $c_n$, $m<n$ or is closing a shortest cycle, which can happen only
in degree $i=l(G)-1$:

\begin{enumerate}
    \item [(1)] If $|s|<l$, then $d_e([G:s])$ has one component less than $[G:s],$ say
    $$c_1, \ldots,c_m \cup e \cup c_n,\ldots c_{n-1}, c_{n+1},\ldots,c_k.$$
    The label of the newly obtained component $c_m \cup e \cup c_n$
    is equal to the product of labels of components being merged, $c_m$ and $c_n$. In other words, $d_e$ is
    given by the multiplication in algebra.
    \item [(2)] If $|s|=l$, then \begin{itemize}
        \item if number of components of s is greater than that of
        $s \cup e$, then $d_e$ is the same as in the case (1).
        \item if the number of components is preserved then
        $d_e([G:s])=(c_1, \ldots,c_m \cup e,\ldots)$ and the closed component $c_m \cup e$ is  decorated by $\Delta (c_m)$.
    \end{itemize}
    \item [(3)] If $|s|>l$, i.e if $p_1([G:s]) \geq 1$, $d_e$ is a zero map.
\end{enumerate}

\end{defn}

Next, in order to have a degree-preserving differential we need to adjust the definition of degrees of basis elements
 of $\A_2 \otimes \A_2$ obtained from comultiplication, according to the convention in Table \ref{H2CoMul}.
In general, the degree would be lowered by the cyclomatic number $p_1(G)$, but since we are closing
the shortest cycle the adjustment is only by $1$.
\begin{table}[h]
\begin{center}
\begin{tabular}{|c|c|}
  \hline
  Basis element & Degree \\
  \hline
  $1 \otimes 1$ & $-1$ \\
  \hline
  $1\otimes x$, $x \otimes 1$ & $0$ \\
 \hline
   $x \otimes x$ & $1$ \\
  \hline
\end{tabular}\end{center}
\caption{Degrees of basis elements in $\A_2 \otimes \A_2$ coming from the comultiplication.}\label{H2CoMul}
\end{table}

The  cohomology ${}^\Delta H^{*,*}(G)$ of the modified bigraded cochain complex
${}^\Delta C(G)$ is also an invariant of all graphs.

Next we analyze the differences between the modified chromatic graph
cohomology and the original one. In general, homology of these
two complexes agree in homological degrees less than the girth of
the graph $l(G)$.

\begin{lemma} For a loopless graph $G$, with $v$
vertices and girth $l(G)=l$, $C^i_{\A_2}(G) \cong C^i_{\A_2}(G)$ for
$0 \leq i < l$. Moreover, there exists an injective map $\alpha:
C^l_{\A_2}(G) \rightarrow {}^{\Delta}C^l_{\A_2}(G)$, so homologies
are isomorphic up to homological level $l-1$.
\end{lemma}
However, we are most interested in the bidegree $(l(G),v-l(G))$,
 in particular, $(2,v-2)$. The change in the
definition is preserving $H^{2,v-2}_{\A_2}(G)$ for loopless graphs even
if multiple edges are allowed. The proof of
this fact relies on duality between the homology and cohomology and the
following Lemma \ref{LemmaH1=H1}.

\begin{lemma}\label{LemmaH1=H1}
For a loopless graph $G$, with possible multiple edges, the
following holds:
$$H_{1,v-2}^{\A_2}(G)\cong {}^{\Delta}H_{1,v-2}^{\A_2}(G)$$
\end{lemma}
\begin{proof}
According to the original and modified definitions of chromatic graph
homology, both chain groups and differentials agree on the zeroth
and first level. Hence, we only need to analyze ${}^{\Delta}d_2$
and ${}^{\Delta}H_{2,v-2}(G)$ if graph $G$ has double or multiple edges.
Under this assumption ${}^{\Delta}C_{2,v-2}(G)$ has more generators than
${}^{\Delta}C_{2,v-2}(G')$, where $G'$ denotes a simple graph obtained from $G$.
Without loss of generality,
denote the double edge by $e=(e_1,e_2)$. We have two different cases:
\begin{enumerate}
    \item If $e$ has weight $x \otimes x$ of degree $1$, then  all but one of the remaining vertices have
    labels $x$. Denote the special vertex by $v$ and the state by $(e(x \otimes x),v(1))$.
    The image of this state ${}^{\Delta}d_2(e,v)=(e_1(x),v(1))\pm
    (e_2(x),v(1))$ gives the relation $(e_1(x),v(1))=
    (e_2(x),v(1))$, so there are no new generators in homology.
    \item If $e$ is labeled by $1 \otimes x$ or $x \otimes 1$, both of degree zero, all of the remaining vertices have to be labeled by $x$ and ${}^{\Delta}d_2(e(1\otimes x))={}^{\Delta}d_2(e(x\otimes
    1))=e_1(1)-e_2(1).$
\end{enumerate}
Therefore $Im{}^{\Delta}d_2$ and $Imd_2$ impose the same relations
on homology which completes the proof. \end{proof}

\begin{cor}\label{CorH1toH2}
For a loopless graph $G$ with $v$ vertices, $\tor H_{\A_2}^{2,v-2}(G) \cong
\tor {}^{\Delta}H^{2,v-2}_{\A_2}(G)$.
\end{cor}

\begin{prop}
  For a connected graph $G$ with  girth $l(G)=1$ homology group  ${}^{\Delta}H_{0,v-1}(G)=0$,
 and  $\tor {}^{\Delta}H^{1,v-1}(G)=0$.
\end{prop}
\begin{proof}
If $e_{\ell}$ denotes a loop in $G$ at the vertex $V$, notice that there is an  epimorphism
$${}^{\Delta}d_1=d_1:{}^{\Delta}C_{1,v-1}(G) \to {}^{\Delta}C_{0,v-1}(G)$$
sending each generator of ${}^{\Delta}C_{1,v-1}(G)$ containing $e_{\ell}$
to a generator of ${}^{\Delta}C_{0,v-1}(G)$ with a label $1$ or $x$ at the vertex $V$, if
$e_{\ell}$ had weight $1 \otimes x$, $x \otimes 1$, or $x \otimes x$.
Hence
${}^{\Delta}H_{0,v-1}^{\A_2}(G) = 0$ and ${\tiny{}^{\Delta}}H^{1,v-1}_{\A_2}(G)$ is torsion free.
\end{proof}

% ${\tiny{}^{\Delta}}$ ${}^{\tiny{\Delta}}$   ${}^{\Delta}{}^{\triangle}{}^{\tiny{\triangle}}$

Finally, we can compute the torsion in Khovanov homology in Proposition \ref{deeperKh},
according to the next Proposition \ref{GraphKhovanovModified}.

\begin{prop}\label{GraphKhovanovModified}
Let $D$ be the diagram of an unoriented framed link $L$, $G=
G_{s_-}(D)$ its associated graph, and $l= l(G)$ girth of a
loopless graph. Then:
\begin{enumerate}
\item[(i)] For all $i < \ell$, we have:
${}^{\Delta}H^{i,j}_{\A_2}(G) \cong H_{a,b}(D),$
 \item[(ii)] For $i=\ell$ we have: $\tor({}^{\Delta}H^{i,j}_{\A_2}(G)) = tor(H_{a,b}(D)),$
\end{enumerate}
where $a=E(G)-2i$, $b=E(G)-2v(G)+4j$ and $H_{a,b}(D)$ are the
Khovanov homology groups of the unoriented framed link $L$ defined
by $D$.
\end{prop}

\begin{prop}\label{deeperKh}
Consider a link diagram $D$ with $n$ crossings. If $D$ is $+$-adequate and $G'=G'_{s_+}(D)$ a simple graph obtained
from $G$, with $p_0(G)^{bi})$ bipartite components, $p_0(G^{nbi})$
non-bipartite components, $p_0(G)$ number of connected components and
the cyclomatic number $p_1(G')$, then
\begin{enumerate}
  \item [(1)] If $G'(D)$ is connected then we get the result (\ref{MainKh})
  \begin{equation*}
  \tor H_{n-4,n+2|D_{s_+}|-8}(D)=
\left\{
\begin{array}{ll}
    \Z_2^{p_1(G'(D))-1}, & \hbox{for $G'(D)$ having an odd cycle;} \\
    \Z_2^{p_1(G'(D))}, & \hbox{{for $G'(D)$ a bipartite graph.}}
\end{array}
\right.
\end{equation*}
  \item [(2)] \begin{equation*}
  \tor H_{n-4,n+2|D_{s_+}|-8}(D)=\Z_2^{p_1(G')+p_0(G)^{nbi})p_1(G)- {p_0(G)^{nbi})+1 \choose 2}}.
\end{equation*}
\end{enumerate}
\end{prop}

\begin{figure}[h]
\begin{center}
  \includegraphics[width=10cm]{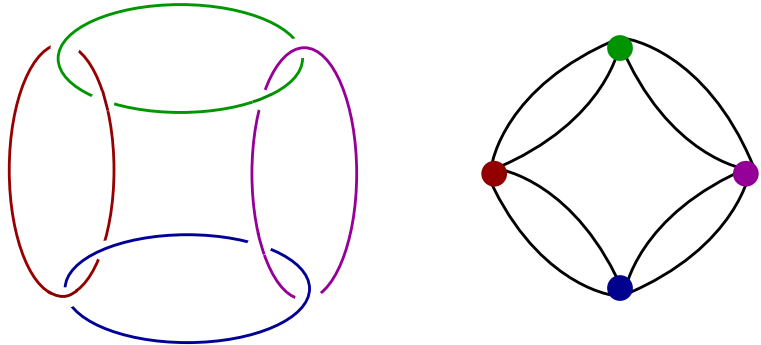}
   \caption{The link $8^4_1$ and the corresponding graph $G_{s_+}(8^4_1).$}\label{Link421}
\end{center}
\end{figure}

\begin{example} The following example illustrates the strength of Proposition \ref{deeperKh} with respect to the previous results.
Consider the link $8^4_1$, shown in Figure \ref{Link421} together with its graph $G(8^4_1)=G_{s_+}(8^4_1)$. Torsion $\tor H_{4,8}(8^4_1)= \Z_2$ in Khovanov homology of this link could not be detected by results of \cite{A-P}, but can be obtained from Theorem \ref{PapCorollary 4.3}(4) together with Proposition \ref{deeperKh}(1). More importantly, it answers the question raised in \cite{A-P}, whether Theorem 3.2 of \cite{A-P} can be improved so that $H_{n-4,n+2|D_{s_+}|-8}(D)$ has $Z_2$ torsion for any $+$-adequate diagram with an even n-cycle ($n\geq 4$).
\end{example}

Our next goal is to find explicit formula for Khovanov homology $H_{n-2i,n+2|D_{s_+}|-4i}(D)$ for $i<\ell(G(D))$,
and $\tor H_{n-2\ell(G(D)),n+2|D_{s_+}|-4\ell(G(D))}(D)$. We plan to use our method for computing elements of a
categorification of skein modules of a product of a surface with the interval as defined in \cite{APS}.

\section{Adequate positive braids}
\label{Braids}

Results obtained in Section \ref{DeeperA2} can be used for finding torsion in
Khovanov homology, in particular we find $2$-torsion for some positive $3$-braids.

\begin{figure}[h]\begin{center}
   \includegraphics[height=5cm]{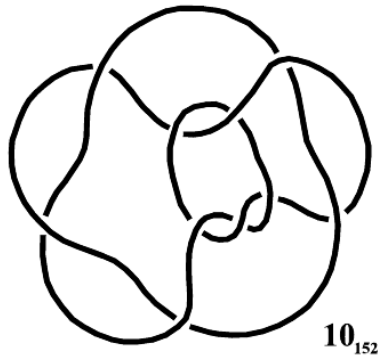}
   \caption{The smallest non-alternating adequate knot $10_{152}.$}\label{FSmallAdequate}
\end{center}
\end{figure}

Notice that the smallest adequate non-alternating knot $10_{152}$ in Rolfsen's table \cite{Ro}
corresponds to the positive minimal braid is $s_1^3s_2^2s_1^2s_2^3$, see Figure \ref{FSmallAdequate}.
The graph assigned to the Kauffman state with all negative resolutions $s_-$ has only multiple edges,
so our method can not detect torsion, see Figure \ref{SmallAdequateNoTor}.

\begin{figure}[h]
\begin{center}
  \includegraphics[width=13cm]{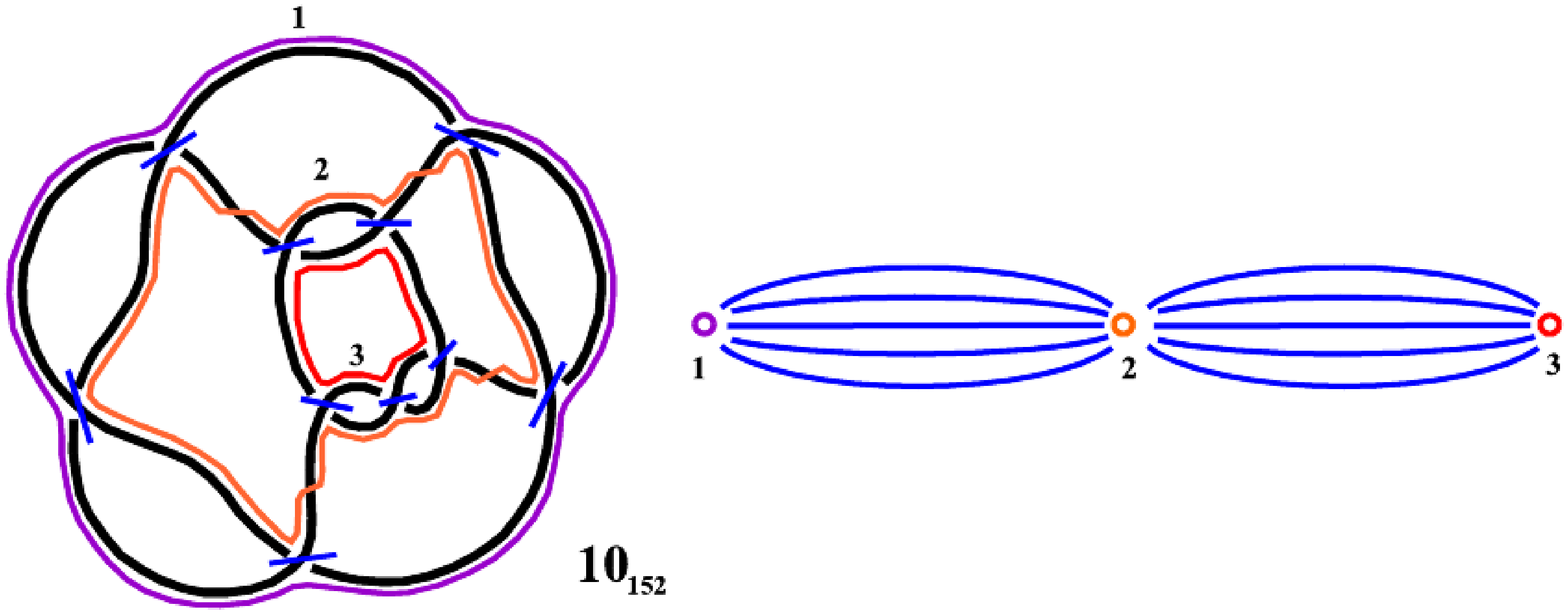}
   \caption{$s_-$ Kauffman state of $10_{152}$ and the corresponding graph.}\label{SmallAdequateNoTor}
\end{center}
\end{figure}

\begin{figure}[h]
\begin{center}
  \includegraphics[width=10cm]{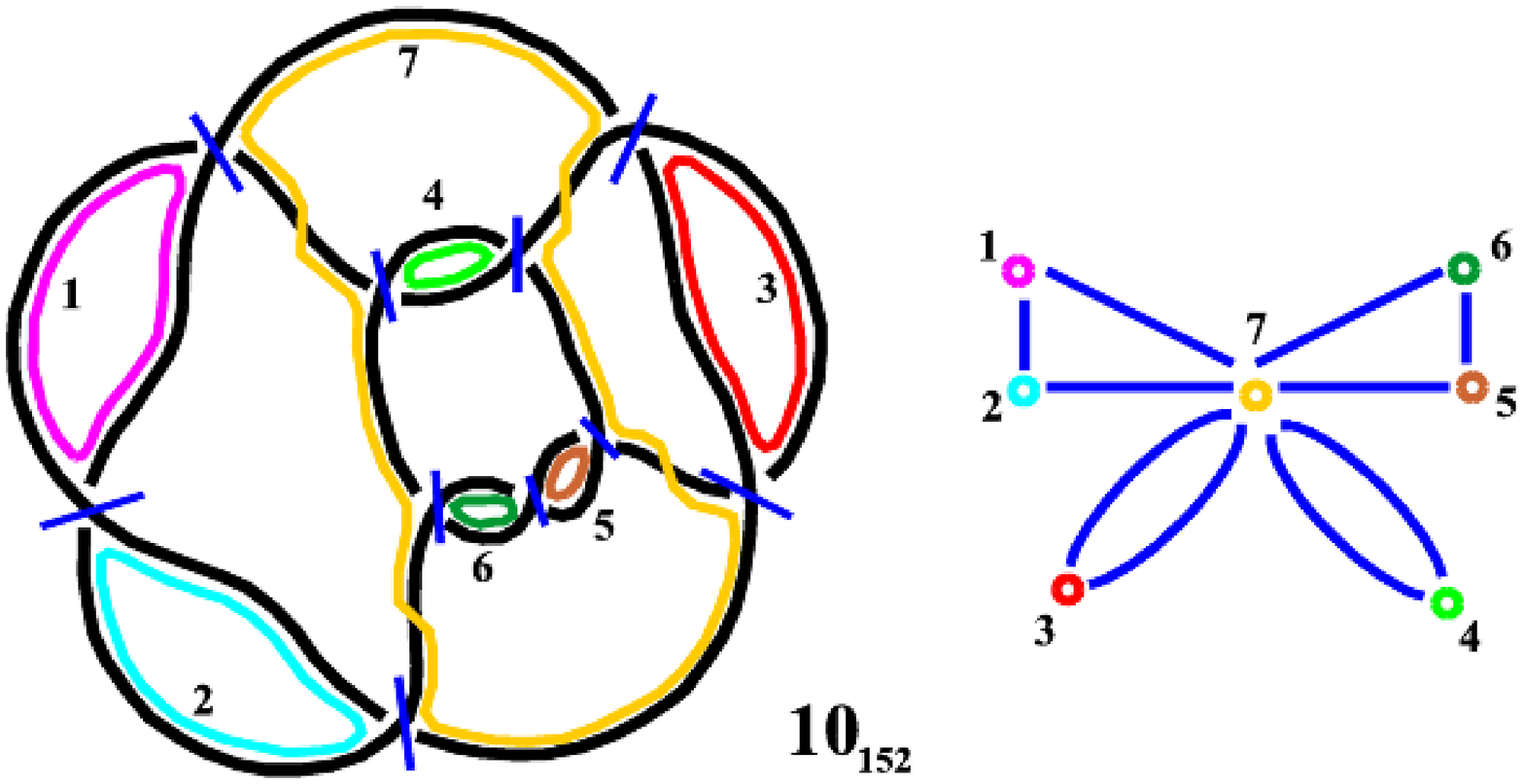}
   \caption{$s_+$ Kauffman state of $10_{152}$ and the corresponding graph.}\label{SmallAdequateTor}
\end{center}
\end{figure}

On the other hand, the graph corresponding to the state $s_+$ with all positive smoothings,
contains triangles, hence, $H_{8,16}\,(10_{152})$ contains $\Z_2$, see Figure  \ref{SmallAdequateTor}. More precisely,
\begin{equation*}
  \tor H_{8,16}(10_{152})= \Z_2= \tor H_{6,12}(10_{152}).
\end{equation*}

This example can be generalized to positive and negative $3$-braids\footnote{$10_{152}$ is a positive braid in the original (old)
convention. In Proposition \ref{braidTri} we use new convention, so $10_{152}$ will have all negative crossings and be a negative knot.}.
In Proposition \ref{braidTri} we state the result for positive braids; the result
for negative
braids is analogous.
\begin{prop}\label{braidTri}
 Let $\gamma= \sigma_{i_1}^{a_1}\sigma_{i_2}^{a_2}...\sigma_{i_k}^{a_k}$ be a positive $3$-braid such that $i_j\in \{1,2\}$,
$ i_j \neq i_{j+1}$, $a_i \geq 1$, and $\widehat{\gamma}$ a closure of $\gamma$. Then the following statements are true
\begin{enumerate}
\item[(1)] The link diagram $\widehat{\gamma}$ is adequate if and only if $a_j \geq 2$ for every $0<j\leq k$.
\item[(2)] If additionally $a_j \geq 3$, for some $j$ then  the link diagram $\widehat{\gamma}$ has $Z_2$-torsion
in Khovanov homology.
\item[(3)] If part (1) holds and $\widehat{\gamma}$ is a knot or a link of 2 components, then (2) holds.
\end{enumerate}
\end{prop}
\begin{proof}
Consider a standard diagram $\widehat{\gamma}$ of a positive $3$-braid $\gamma$. Since the diagram is positive, link is $+$-adequate. In this case graph $G_{s_+}$ has three vertices and only $2$-cycles, compare with Figure \ref{SmallAdequateNoTor}. On the other hand, the graph $G_{s_-}$ contains an $a_i$-gon for any $0<i \leq k$. In particular, if all $a_i \geq 2$, this graph has no loops, hence $\widehat{\gamma}$ is $-$-adequate. Furthermore, if at least one $a_j \geq 3 $, than the girth of the corresponding graph $G_{s-}$ is at least $3$. According to Theorem \ref{Theorem 3.1} and  Theorem \ref{PapCorollary 4.3} Khovanov homology of such $3$-braids contains $Z_2$-torsion. Part (3) follows from the fact that when all $a_i=2$, then $\widehat{\gamma}$ is a link of $3$ components.
\end{proof}

The weaker version of Proposition \ref{braidTri} holds for all positive $n$-braids.
\begin{prop}\label{Any braid} Let $\gamma= \sigma_{i_1}^{a_1}\sigma_{i_2}^{a_2}...\sigma_{i_k}^{a_k}$
be a positive n-braid such that $a_i \geq 2$ for any $i$ then $\widehat{\gamma}$ is an adequate diagram.
If additionally $a_j \geq 3$ for some $j$ then $\widehat{\gamma}$ has $Z_2$-torsion in Khovanov homology.
\end{prop}

\section{Conjectures}\label{conj}

The main goal of this paper was to enhance our understanding of torsion in Khovanov homology. In order to do so, we have
analyzed those gradings in chromatic graph cohomology that agree with Khovanov homology.
This approach brought new insights about torsion that agree with the recent results by A.~Shumakovitch \cite{Sh3}
stating that there is no other torsion except $\Z_2$ in Khovanov homology of alternating knots.
Experimental results obtained using Shumakovitch's software \emph{KhoHo}, \emph{Knotscape} by M. Thistlethwaite,
and \emph{LinKnot} by S.~Jablan and the second author show that there are eight positive $15$-crossing
knots whose $4$-braid diagrams are adequate, and which have $Z_4$ torsion in
Khovanov homology.\footnote{ Closure of following braids have $Z_4$ torsion in Khovanov homology BR[4,{1,1,2,2,1,1,3,2,2,2,1,3,2,2,3}]
BR[4,{1,1,2,2,2,1,1,3,2,2,1,3,2,2,3}]
BR[4,{1,2,2,1,3,2,2,2,1,3,2,2,2,3,3}]
BR[4,{1,2,2,1,3,3,3,2,2,2,1,3,2,2,3}]
BR[4,{1,2,2,1,3,2,2,2,2,2,1,3,2,2,3}]
BR[4,{1,2,2,1,3,2,2,2,1,3,2,2,3,3,3}]
BR[4,{1,2,2,1,3,2,2,2,1,3,2,2,2,2,3}]
BR[4,{1,2,2,2,1,3,2,2,2,1,3,2,2,2,3}], as verified by Cotton Seed, Slavik Jablan, and Alexander Shumakovitch \cite{Se,Ja,Sh}. }.
 We suspect that the order of torsion in Khovanov homology partially
depends on the minimal braid index of a given link as stated in the following conjecture.

\begin{conjecture}{PS braid conjecture}\,
\begin{itemize}
  \item[(1)] Khovanov homology of a closed $3$-braid can have only $Z_2$ torsion.
  \item[(2)] Khovanov homology of a closed $4$-braid cannot have an odd torsion.
  \item[(2')] Khovanov homology of a closed $4$-braid can have only $Z_2$ and $Z_4$ torsion.
  \item[(3)] Khovanov homology of a closed $n$-braid cannot have p-torsion for $p>n$ (p prime).
  \item[(3')] Khovanov homology of a closed $n$-braid cannot have $Z_{p^r}$ torsion for $p^r>n$.
\end{itemize}
\end{conjecture}

Note that we are stating these conjectures with various degrees of confidence. The case of
$3$-braids were extensively tested using A.~Shumakovitch software \emph{KhoHo}, and P.~Turner proved that
Khovanov homology of $(3,q)$ torus links can contain only $2$-torsion \cite{Lo,Tu}. In 2011, W.~Gilliam \cite{Gi} showed that only $Z_2$ torsion is possible in their Khovanov homology. D.~Bar-Natan \cite{BN} checked that $(n,4)$ torus knots have $Z_4$ torsion for
$n=5,7,9,11$. % and additionally we found positive adequate $4$-braid $s_1^2s_2^2s_1^2s_2s_3^2s_2^2s_1^3s_2^3s_3^3s_2$.
%Jozef by modifying T(5,4)

\begin{example}\label{TorAlex}
  Till summer of 2012, only $5$-strand torus knots: ($6$,$5$), ($7$,$5$), ($8$,$5$) and ($9$,$5$) were known to have $Z_5$-torsion in Khovanov homology. We predicted that positive adequate $36$-crossing knot $K$ given by the closure of the braid $s_1^2s_2^2s_1^3s_2^2s_1s_3s_2^2s_4^2s_3s_1^2s_2^2s_1^3s_2^3s_1^2s_3s_2^2s_4^3s_3^2$ has $Z_5$ torsion, which was confirmed by A. Shumakovitch using \emph{JavaKh} \cite{BNG}. More precisely, we show that  homology mod $5$ and mod $7$ have different ranks
      while homology mod $7$ and mod $11$ have the same ranks. The difference between Khovanov polynomials computed mod $5$ and mod $7$ (see Appendix) is strictly positive:
$$KH_5(K)-KH_7(K)=(t^{12} + t^{11})q^{51} + (t^{11} + t^{10})q^{47}.$$
This means that the rank of $5$-torsion is strictly greater than one of $7$-torsion, hence torsion of order $5$ exists at least in degrees $(12,51) $, and  $(11,47)$\footnote{Theoretically, Khovanov homology can contain more $5$-torsion, but then
it must coincide with $7$-torsion, which we predict to be trivial.}.
\end{example}
 Finally for ($8$,$7$) torus knot Bar-Natan computed Khovanov homology and showed that it
contains $Z_7$, $Z_5$, $Z_4$, and $Z_2$-torsion but this $48$ crossing $7$-braid reaches the limits of current
computational resources.\\

\noindent ACKNOWLEDGEMENTS\\
J.H.~Przytycki  was partially supported by the  NSA-AMS 091111 grant and by NSF-DMS-1137422 grant.
R.~Sazdanovi\'c was fully supported by the Postdoctoral Fellowship at MSRI, Berkeley 
during the early stages of this project, and NSF 0935165 and AFOSR FA9550-09-1-0643 grants towards the end.
We are profoundly thankful to  the Mathematisches Forschungsinsitut Oberwolfach for providing us with
   unique computer facilities that made computations of
   Example~\ref{TorAlex} possible.\\

%%%%%%%%%%%%%%%%%%%%%%
%%
%%               REFERENCES
%%
%%%%%%%%%%%%%%%%%%%%%%

\appendix
\section{Khovanov homology computations}

We include Khovanov polynomials of positive adequate $36$-crossing knot $K$ given by the closure of the braid $s_1^2s_2^2s_1^3s_2^2s_1s_3s_2^2s_4^2s_3s_1^2s_2^2s_1^3s_2^3s_1^2s_3s_2^2s_4^3s_3^2$ computed over $Z_5$ and $Z_7$. Computations were done in    \emph{JavaKh} \cite{BNG} by A. Shumakovitch using Mathematisches Forschungsinsitut Oberwolfach world-class computer facilities.
{\small
\begin{eqnarray*}
&& KH_5(K)= q^{31}t^0 + q^{33}t^{0} + q^{35}t^{2} + q^{39}t^{3 } + 2q^{37}t^{4} + q^{39}t^{4 } + 2+q^{41}t^{5} + q^{43}t^{5} +  q^{39}t^{6 }+
 \\ & &   2q^{41}t^{6 } +2q^{43}t^{7 } + 2q^{45}t^{7 } +4q^{41}t^{8 } + 3q^{43}t^{8}+  q^{47}t^{8 } + 13q^{43}t^{9 } + 4q^{45}t^{9 } + 4q^{47}t^{9 } +
 2q^{43}t^{10 } +\\ & &  29q^{45}t^{10 } + 14q^{47}t^{10}  + q^{51}t^{10 } + 9q^{45}t^{11 } + 44q^{47}t^{11 } +31q^{49}t^{11} + q^{51}t^{11 } + 2q^{45}t^{12 } +34q^{47}t^{12 } +
 \\ & &  68q^{49}t^{12 } + 42q^{51}t^{12 } + 2q^{53}t^{12 } + 11q^{47}t^{13 } + 85q^{49}t^{13 } + 97q^{51}t^{13 } + 59q^{53}t^{13 } + 45q^{49}t^{14 } +
 \\ & &  159q^{51}t^{14 } +142q^{53}t^{14 } + 63q^{55}t^{14 } + 137q^{51t^15 } + 245q^{53}t^{15 } + 202q^{55}t^{15 } +9q^{57}t^{15 } + 345q^{53}t^{16 } +
 \\ & & 5376q^{55}t^{16 } +237q^{57}t^{16 } + 54q^{59}t^{16 } + 735q^{55}t^{17 } +589q^{57}t^{17 } + 260q^{59}t^{17 } + 37q^{61}t^{17 } +
 \\ & &1328q^{57}t^{18 } +953q^{59}t^{18 } + 253q^{61}t^{18 } +21q^{63}t^{18 } + 2040q^{59}t^{19 } + 1501q^{61}t^{19 } + 220q^{63}t^{19 } +
 \\ & &9q^{65}t^{19 } + 2729q^{61}t^{20 } + 2149q^{63}t^{20 } + 173q^{65}t^{20 } + 2q^{67}t^{20 } + 2q^{61}t^{21 } + 3203q^{63}t^{21 } +
 \\ & &2779q^{65}t^{21 } +109q^{67}t^{21 } +11q^{63}t^{22 } +3344q^{65}t^{22 } + 3219q^{67}t^{22 } + 50q^{69}t^{22 } + 36q^{65}t^{23 } +
 \\ & &3127q^{67}t^{23 } +3345q^{69}t^{23 } + 16q^{71}t^{23 } +81q^{67}t^{24 } +2608q^{69}t^{24 } + 3116q^{71}t^{24 } +3q^{73}t^{24 }+
 \\ & & 137q^{69}t^{25 } +1934q^{71}t^{25}  + 2572q^{73}t^{25 } +191q^{71}t^{26 } +1271q^{73}t^{26 } +1853q^{75}t^{26 } + 228q^{73}t^{27 } +
 \\ & & 759q^{75}t^{27 } +1134q^{77}t^{27 } + 238q^{75}t^{28 } +446q^{77}t^{28 } +568q^{79}t^{28 } + 219q^{77}t^{29 } + 294q^{79}t^{29 } +
  \\ & & 218q^{81}t^{29 } +175q^{79}t^{30 } + 226q^{81}t^{30 } +56q^{83}t^{30 } + 119q^{81}t^{31 } +175q^{83}t^{31 } + 7q^{85}t^{31 } + 65q^{83}t^{32 } +
  \\ & & 119q^{85}t^{32 }+26q^{85}t^{33 } + 65q^{87}t^{33 } + 7q^{87}t^{34 } + 26q^{89}t^{34}  +q^{89}t^{35 } + 7q^{91}t^{35} + q^{93}t^{36}    \end{eqnarray*}

\begin{eqnarray*}
KH_7(K) &=&q^{31}t^{0} + q^{33}t^{0} + q^{35}t^{2} + q^{39}t^{3 } + 2q^{37}t^{4} + q^{39}t^{4 } +2q^{41}t^{5} + q^{43}t^{5} +q^{39}t^{6 } + 2q^{41}t^{6 } +
\\ & &  2q^{43}t^{7 } + 2q^{45}t^{7 } +4q^{41}t^{8 } + 3q^{43}t^{8} + q^{47}t^{8 } + 13q^{43}t^{9 } +4q^{45}t^{9 } + 4q^{47}t^{9 } +
\\ & &  2q^{43}t^{10 } + 29q^{45}t^{10 } + 13q^{47}t^{10} + q^{51}t^{10 } + 9q^{45}t^{11 } + 43q^{47}t^{11 } +31q^{49}t^{11 } +
\\ & &  2q^{45}t^{12 } + 34q^{47}t^{12 } + 68q^{49}t^{12 } + 41q^{51}t^{12 } + 2q^{53}t^{12 } +11q^{47}t^{13 } + 85q^{49}t^{13 } +
\\ & & 97q^{51}t^{13 } + 59q^{53}t^{13 } + 45q^{49}t^{14 } +159q^{51}t^{14 } +142q^{53}t^{14} +63q^{55}t^{14 } + 137q^{51}t^{15 } +
\\ & & 245q^{53}t^{15 } + 202q^{55}t^{15 } + 59q^{57}t^{15 } +345q^{53}t^{16 } + 376q^{55}t^{16 } +237q^{57}t^{16 } +
\\ & & 54q^{59}t^{16 } + 735q^{55}t^{17 } + 589q^{57}t^{17 } +260q^{59}t^{17 } + 37q^{61}t^{17 } + 1328q^{57}t^{18 } +
\\ & & 953q^{59}t^{18 } + 253q^{61}t^{18 } + 21q^{63}t^{18 } +2040q^{59}t^{19 } + 1501q^{61}t^{19 } + 220q^{63}t^{19 } +
\\ & & 9q^{65}t^{19 } +2729q^{61}t^{20 } + 2149q^{63}t^{20 } +173q^{65}t^{20 } + 2q^{67}t^{20 } + 2q^{61}t^{21 } + 3203q^{63}t^{21 } +
\\ & & 2779q^{65}t^{21 } +109q^{67}t^{21 } +11q^{63}t^{22 } + 3344q^{65}t^{22 } + 3219q^{67}t^{22 } + 50q^{69}t^{22 } +
\\ & & 36q^{65}t^{23 } + 3127q^{67}t^{23 } +3345q^{69}t^{23 } + 16q^{71}t^{23 } + 81q^{67}t^{24 } + 2608q^{69}t^{24 } +
\\ & & 3116q^{71}t^{24 } +3q^{73}t^{24 } +137q^{69}t^{25 } + 1934q^{71}t^{25 } + 2572q^{73}t^{25 } + 191q^{71}t^{26 } +
\\ & & 1271q^{73}t^{26 } + 1853q^{75}t^{26 } +228q^{73}t^{27 } +1134q^{77}t^{27 } + 238q^{75}t^{28 } + 446q^{77}t^{28 } +
\\ & & 568q^{79}t^{28 } +219q^{77}t^{29 } +294q^{79}t^{29 } +218q^{81}t^{29 } +759q^{75}t^{27 } + 175q^{79}t^{30 } + 226q^{81}t^{30 } +
\\ & &  56q^{83}t^{30 } +119q^{81}t^{31 } + 175q^{83}t^{31 } +7q^{85}t^{31 } + 65q^{83}t^{32 } +119q^{85}t^{32 } + 26q^{85}t^{33 } +
\\ & &65q^{87}t^{33 } + 7q^{87}t^{34 } + 26q^{89}t^{34}  +
q^{89}t^{35 } + 7q^{91}t^{35} + q^{93}t^{36}\end{eqnarray*}
}

\noindent Jozef H. Przytycki \\ \noindent \textsc{Dept. of Mathematics,
The George Washington University,\\ Washington, DC 20052}\\
e-mail: {\tt przytyck@gwu.edu}, \ \\
and Institute of Mathematics, University of Gda\'nsk.\\

\noindent Radmila Sazdanovic\\ \noindent \textsc{Dept. of Mathematics
University of Pennsylvania,\\
Philadelphia, PA 19104-6395},\\
e-mail: {\tt radmilas@gmail.com}
\end{document}